\documentclass{amsart}

\usepackage{amsmath, graphicx, cite, enumerate, comment}
\usepackage{amssymb, amsthm, amscd}

\newtheorem{theorem}{Theorem}[section]

\newtheorem{lemma}[theorem]{Lemma}
\newtheorem{proposition}[theorem]{Proposition}
\newtheorem{corollary}[theorem]{Corollary}

\theoremstyle{definition}

\theoremstyle{remark}
\newtheorem{remark}[theorem]{Remark}

\numberwithin{equation}{section}

\newcommand{\abs}[1]{\left|#1\right|}

\newcommand{\Rm}{\textup{Rm}}
\newcommand{\Ric}{\textup{Ric}}

\newcommand{\Vol}{\textup{Vol}}
\newcommand{\diam}{\textup{diam}}

\newcommand{\FS}{\textup{FS}}
\newcommand{\Tr}{\textup{Tr}}

\newcommand{\eval}[2]{\left. #1 \right|_{#2}}

\newcommand{\R}{\mathbb{R}}
\newcommand{\C}{\mathbb{C}}
\newcommand{\Q}{\mathbb{Q}}

\newcommand{\CP}{\mathbb{CP}}

\newcommand{\dd}[1]{\frac{\partial}{\partial #1}}
\newcommand{\D}[2]{\frac{\partial #1}{\partial #2}}
\newcommand{\ddbar}{\partial \bar{\partial}}

\def\<{\langle}
\def\>{\rangle}
\def\({\left(}
\def\){\right)}

\def\p{\partial}

\def\Ric{{\rm Ric}}

\begin{document}
\title[The collapsing rate of KRF with infinite time singularity]
{The Collapsing Rate of the K\"ahler-Ricci Flow with Regular 
Infinite Time Singularity}
\author[Frederick T.H. Fong]{Frederick Tsz-Ho Fong$^*$}
\address{$^*$Department of Mathematics, Stanford University, 450 Serra Mall, Stanford CA 94305, USA}
\email{thfong@math.stanford.edu}
\thanks{$^*$Partially supported by US NSF Grant: DMS-\#1105323}

\author[Zhou Zhang]{Zhou Zhang$^\dag$}
\address{$^\dag$Carslaw Building, School of Mathematics and Statistics, Sydney University
NSW 2006, Australia}
\email{zhangou@maths.usyd.edu.au}
\thanks{$^\dag$Partially supported by Australian Research Council Discovery Project: DP110102654}

\date{\today}

\begin{abstract}
We study the collapsing behavior of the K\"ahler-Ricci flow on a compact 
K\"ahler manifold $X$ admitting a holomorphic submersion $X \xrightarrow
{\pi}\Sigma$ inherited from its canonical bundle, where $\Sigma$ is a K\"ahler 
manifold with $\dim_\C \Sigma<\dim_\C X$. We show that the flow metric 
degenerates at exactly the rate of $e^{-t}$ as predicted by the cohomology 
information, and so the fibres $\pi^{-1}(z)$, $z \in \Sigma$ collapse at the 
optimal rate $\textup{diam}_t (\pi^{-1}(z)) \simeq e^{-t/2}$. Consequently, it leads to some analytic and geometric extensions to the regular case of  Song-Tian's works \cite{ST07, ST08}. Its applicability to general Calabi-Yau fibrations will also be discussed in local settings.
\end{abstract}
\maketitle

\section{Introduction}
In this note, we let $X$ be a closed connected K\"ahler manifold with $\dim_\C X = n$ which admits the following fibration. Let $(\Sigma, \omega_\Sigma)$ be a K\"ahler manifold with $\dim_\C
\Sigma = n-r<n$ and $X \xrightarrow{\pi} \Sigma$ is a surjective holomorphic 
submersion. This submersion gives a smooth fibration structure by classical results due to Ehresmann \cite{Eh51} and Fischer-Grauert \cite{FG65}. For each $z \in \Sigma$, we call $\pi^{-1}(z)$ a fibre based at $z$, which is a complex submanifold of $X$ with $\dim_\C = r$. $X$ is a smooth fibre bundle over $\Sigma$, but the induced complex structure on each fibre may vary. In the case where the fibres are isomorphic, $X$ is a holomorphic fibre bundle over $\Sigma$. Here, we allow $\Sigma$ to be a point, i.e. $r=n$.

Throughout the note, we assume that the first Chern class $c_1(X)= -\pi^*\alpha$ for some K\"ahler class $\alpha$ on $\Sigma$, and so each fibre $\pi^{-1}(z)$ is a Calabi-Yau manifold. We consider the following normalized K\"ahler-Ricci flow on $X$, defined by
\begin{equation}
\label{KRF}
\frac{\p \omega_t}{\p t}=-\Ric\(\omega_t\)-\omega_t, \quad
\eval{\omega_t}{t=0}=\omega_0, 
\end{equation}
with any K\"ahler metric $\omega_0$ as the initial metric.

The K\"ahler class $[\omega_t]$ at time $t$ is precisely given by $-c_1(X)+e^{-t}([\omega_0] +c_1(X))$, where we have chosen 
the convention $c_1(X)=[\Ric(\omega)]$ for any K\"ahler metric 
$\omega$ on $X$. The maximal existence time $T$ of \eqref{KRF} 
is uniquely determined by the optimal existence result due to Tian and the second-named author 
in \cite{TZ06}, namely 
\[T= \sup\{t:-c_1(X)+e^{-t}([\omega_0] +c_1(X))~\text{is K\"ahler.}\}.\]

The infinite time singularity case (i.e. $T = \infty$) in this note is as follows. We have 
a surjective holomorphic submersion $\pi$ as described above.
Moreover, $\pi^*[\omega_\Sigma]=-m\cdot c_1(X)$ for 
some K\"ahler class $[\omega_\Sigma]$ over $\Sigma$ and a positive 
integer $m$. In practice, we usually have $\pi$ generated by holomorphic 
sections of the line bundle $m\cdot K_X$ as a map $X \xrightarrow{\pi} \CP^N$, where $K_X$ is the canonical bundle of $X$, i.e. $c_1(K_X)=-c_1(X)$, and $\Sigma$ is the image of $\pi$. One can take $\omega_\Sigma=\omega_{\FS}|_\Sigma$ where $\omega_{\FS}$ is the Fubini-Study metric on $\CP
^N$, and $[\omega_\Sigma]$ is the restriction of the hyperplane class 
of $\CP^N$ to $\Sigma$. Under this setting, $-c_1(X)$ is semi-ample and by the optimal existence 
result, the flow exists forever. The limiting K\"ahler class as $t \to \infty$ is exactly $-c_1(X)$. We call this {\bf regular infinite time 
singularity}.

Define $\omega_\infty=\pi^*\omega_\Sigma$ and set  
$$\hat\omega_t=\omega_\infty+e^{-t}(\omega_0-\omega_\infty).$$ 
Then $\hat\omega_t$ is a reference metric in the same K\"ahler class as the flow metric $\omega_t$.
The following is the main result of this paper:

\begin{theorem}
\label{main}
Let $X \xrightarrow{\pi} \Sigma$ be a holomorphic submersion described above and $\omega_t$ satisfies the normalized K\"ahler-Ricci flow $\partial_t \omega_t = -\Ric(\omega_t) - \omega_t$ on $X$. Assume we have regular infinite time singularity and the K\"ahler class $[\omega_t]$ limits to $\pi^*[\omega_\Sigma]$ for some K\"ahler metric $\omega_\Sigma$ on $\Sigma$ (i.e. $c_1(K_X) = \pi^*[\omega_\Sigma]$). Then, using the notations introduced above, we have $$C^{-1}\hat\omega_t\leq \omega_t\leq C\hat
\omega_t$$ where $C$ is a uniform constant depending only on $n, r, 
\omega_0$, and $\omega_\Sigma$. Hence, $\omega_t\simeq e^{-t}
\omega_0$ along fibres and the fibres have diameters uniformly 
bounded from above and below by exponentially decaying terms, i.e.
$$C^{-1}e^{-\frac{t}{2}} \leq \diam_t (\pi^{-1}(z)) \leq C e^{-\frac{t}{2}}, \quad \text{ for any } z \in \Sigma.$$
\end{theorem}

This result shares the same theme with several related works in the current literature. In \cite{ST07, ST08}, Song and Tian studied the collapsing behavior of elliptic and Calabi-Yau fibrations with non-big semi-ample canonical bundle under the normalized K\"ahler-Ricci flow \eqref{KRF}, and showed that the metric on the regular part converges, as a current, to a generalized K\"ahler-Einstein metric on the base manifold (see also \cite{KT08}). In case of elliptic fibrations, it was proved in \cite{ST07} that the convergence is in $C^{1,\alpha}$-sense for any $\alpha < 1$ on the potential level. Theorem \ref{main} in this note asserts if the fibration is regular then one can obtain an optimal fibre-collapsing rate $\diam_t \simeq e^{-t/2}$, and more importantly, shows that the $C^{1,\alpha}$-convergence also holds for smooth Calabi-Yau fibrations of general dimensions (see Corollary \ref{cor:C_one_alpha}).

There are analogous collapsing results for the unnormalized K\"ahler-Ricci flow $\partial_t \omega_t = -\Ric(\omega_t)$ with finite time singularity. For instance, the collapsing behavior of $\CP^r$-bundles was studied by Song, Sz\'ekelyhidi and Weinkove in \cite{SW09} and \cite{SSW11} (see also \cite{F1} by the first-named author). The collapsing behavior of Ricci-flat metrics on Calabi-Yau manifolds is also studied in \cite{To10} and \cite{GTZ11} by Gross, Tosatti and Y. Zhang. The common theme shared by all the aforesaid works is that the limiting behavior of the K\"ahler metric can be read off by the cohomological data.

Inspired by \cite{GTZ11}, we deduce several geometric and analytic consequences of Theorem \ref{main} on toric fibrations, a special case of Calabi-Yau fibrations with complex tori as fibres. The existence of semi-flat forms on toric fibrations with a good rescaling property allows us to make use of Theorem \ref{main} to further strengthen the $C^{1,\alpha}$-convergence. Using a parabolic analogue of Gross-Tosatti-Y.Zhang's argument, we show that on toric fibrations if the initial K\"ahler class is rational, then along the K\"ahler-Ricci flow we have (see Propositions \ref{prop:type_III}, \ref{prop:C_infty} and \ref{prop:fiber}):
\begin{enumerate}[(i)]
\item the Riemann curvature $\|\Rm\|_{\omega_t}$ is uniformly bounded;
\item $\omega_t$ converges smoothly to a generalized K\"ahler-Einstein metric on $\Sigma$; and
\item when restricted to each torus fibre, $e^t \omega_t$ converges smoothly to a flat metric on the fibre.
\end{enumerate}
Some of the above statements, particularly (ii), were conjectured in \cite{ST07, ST08} (see also \cite{SW_KRF}) on regular Calabi-Yau fibrations, and on general Calabi-Yau fibrations away from singular fibres. A recent preprint \cite{Gil12} by Gill gives an affirmative answer to the case where $X$ is a Cartesian product of a complex torus and a compact K\"ahler manifold with negative first Chern class. Our results hence further affirm these conjectures on a wider class of regular toric fibrations. One fundamental assumption in Propositions \ref{prop:type_III}, \ref{prop:C_infty} and \ref{prop:fiber} is that the initial K\"ahler class is rational. It guarantees the existence of a suitable semi-flat form explicitly constructed by Gross-Tosatti-Y.Zhang in \cite{GTZ11}. We hope that this technical assumption can be removed.

\subsection*{Acknowledgements}
The first-named author would like to thank his advisor Richard Schoen for his constant encouragement and support throughout the years in Stanford University. He would also like to thank Yanir Rubinstein, Jian Song and Ben Weinkove for many valuable discussions and inspiring ideas.

The second-named author would like to thank Gang Tian for introducing him into this interesting research area and constant support. He also like to thank the School of Mathematics and Statistics at Sydney University for providing the great research environment.

Both authors would like to thank Valentino Tosatti especially for suggesting the use of arguments developed in \cite{GTZ11} which contributes to a great part of Section 5, and also for his insightful comments on Section 6 in our previous draft. They also thank the referee for the careful check and suggestions. 

\section{Some Estimates on Decay Rates}

In this section, we prove the necessary estimates for establishing Theorem 
\ref{main}. We adopted the techniques developed in \cite{TZ06, ST07, Z09, To10} etc. Once the pointwise decay of the volume form $\omega_t^n$ is established, the rest of the argument will follows similarly as in \cite{F2} by the first-named author (see also \cite{To10} for an elliptic analogue of the argument).

We rewrite the K\"ahler-Ricci flow \eqref{KRF} as a parabolic complex 
Monge-Amp\`ere equation in the same way as in \cite{TZ06, ST07} etc. We use the family of reference metrics $\hat
{\omega}_t$ defined before, which is in the same K\"ahler class 
as $\omega_t$. By the $\ddbar$-lemma, there exists a family of 
smooth functions $\varphi_t$ such that $\omega_t = \hat\omega_t
+\sqrt{-1}\ddbar\varphi_t$. Let $\Omega$ be a volume form on $X$ 
such that
\begin{equation}
\label{Omega}
\sqrt{-1}\ddbar\log\Omega =\omega_\infty=\pi^*\omega_\Sigma,
\end{equation}
whose existence is clear from the cohomology consideration. 

Then it is easy to check that the K\"ahler-Ricci flow \eqref{KRF} is 
equivalent to the following scalar evolution equation (with a complex 
Monge-Amp\`ere looking): 
\begin{equation}
\label{MAP}
\D{\varphi_t}{t} = \log \frac{(\hat\omega_t + \sqrt{-1}\ddbar\varphi_t)
^n}{e^{-rt} \Omega}-\varphi_t, ~~~~\varphi_0=0,
\end{equation}
and so the solution $\varphi_t$ also exists forever.

\vspace{0.1in}

\noindent{\bf Convention:} in this note, we denote $C > 0$ to be a uniform 
constant which depends only on $n, r, \omega_0, \omega_\Sigma$, 
and may change from line to line. $\Delta$ stands for Laplacian with 
respect to the flow metric $\omega_t$. 

\vspace{0.1in}

We begin with the following $0$th-order estimates. 

\begin{lemma}
\label{ddt_phi:bbd1}
For \eqref{MAP}, there exists a uniform constant $C = C(n, r, \omega_0, 
\omega_\Sigma)$ such that
$$|\varphi_t|\leq C, ~~~~~~~~~~~~\vline\frac{\p \varphi_t}{\p t}
\vline \leq C$$
\end{lemma}

\begin{proof}

Because $\pi: X\to \Sigma$ is a fibre bundle structure and that $\hat\omega_t^n \simeq e^{-rt}\Omega$, by a straightforward Maximum Principle 
argument, we have $|\varphi_t|\leq C$.

Next we derive the bound for $\frac{\partial \varphi_t}{\partial t}$. Taking $t$-derivative of \eqref{MAP} we get 
$$\frac{\partial}{\partial t}\(\frac{\partial \varphi_t}{\partial t}\)=\Delta
\(\frac{\partial \varphi_t}{\partial t}\)-e^{-t}\Tr_{\omega_t}(\omega_0-
\omega_\infty)-\frac{\partial \varphi_t}{\partial t}+r.$$

We can also reformulate it to the following two equations: 
$$\frac{\partial}{\partial t}\(e^t\frac{\partial \varphi_t}{\partial t}\)
=\Delta\(e^t\frac{\partial \varphi_t}{\partial t}\)-\Tr_{\omega_t}
(\omega_0-\omega_\infty)+re^t,$$
\begin{equation}\label{eq:box_dphi+phi}
\frac{\partial}{\partial t}\(\frac{\partial \varphi}{\partial t}+\varphi
_t\)=\Delta\(\frac{\partial \varphi_t}{\partial t}+\varphi_t\)-n+r+\Tr_
{\omega_t}\omega_\infty.
\end{equation}

The difference of these two is 
$$\frac{\partial}{\partial t}\left((e^t-1)\frac{\partial \varphi_t}{\partial 
t}-\varphi_t\right)=\Delta\left((e^{t}-1)\frac{\partial \varphi_t}{\partial 
t}-\varphi_t\right)-\Tr_{\omega_t}\omega_0+re^t+n-r.$$

Applying Maximum Principle and the bounds for $\varphi_t$, we have
$$\frac{\partial \varphi_t}{\partial t}\leq\frac{(n-r)t+re^t+C}
{e^t-1}\leq C.$$

For the lower bound, we can mimic the argument in \cite{ST07} as follows. 
$$n^{-n}\Tr_{\omega_t}\hat\omega_t \geq\frac{\hat\omega^n_t}
{\omega^n_t}=\frac{\hat\omega^n_t}{e^{\frac{\partial \varphi_t}
{\partial t}+\varphi_t-rt}\Omega}\geq Ce^{-\frac{\partial \varphi_t}
{\partial t}}.$$

We can then combine   
$$\left(\frac{\partial}{\partial t}-\Delta\right)\(\frac{\partial \varphi_t}
{\partial t}+\varphi_t\)=-n+r+\Tr_{\omega_t}\omega_\infty\geq 
-n+r$$
$$\left(\frac{\partial}{\partial t}-\Delta\right)\varphi_t=\frac{\partial 
\varphi_t}{\partial t}-n+\Tr_{\omega_t}\hat\omega_t\geq\frac
{\partial \varphi_t}{\partial t}-n+Ce^{-\frac{\partial \varphi_t}
{\partial t}}.$$
to arrive at  
$$\left(\frac{\partial}{\partial t}-\Delta\right)\(\frac{\partial \varphi_t}
{\partial t}+2\varphi_t\) \geq\frac{\partial \varphi_t}{\partial t}-C
+Ce^{-\frac{\partial \varphi_t}{\partial t}}.$$
Again applying Maximum Principle and the bounds of $\varphi_t$, we 
can conclude the lower bound for $\frac{\partial \varphi_t}{\partial t}$.
\end{proof}
\begin{remark}
For the unnormalized K\"ahler-Ricci flow with finite time singularity, the first-named author has to assume in \cite{F2} a uniform bound on $\Tr_{\omega_0}\Ric(\omega_t)$ in order to derive an appropriate pointwise decay of the volume form $\omega_t^n$. Note that such an assumption is not needed in the setting of this note.

In \cite{ST11}, there is a delicate argument to establish the same results as in Lemma \ref{ddt_phi:bbd1} when the  $\pi$ is not assumed to be regular.
\end{remark}

These $0$th-order bounds provide the exact setting as in \cite{Z10} and \cite{ST11}, and lead to a sequence of estimates which eventually prove the uniform bound of the scalar curvature. Among those estimates, there is one which is useful for our purpose of this note:
\begin{equation}
\label{schwarz}
\Tr_{\omega_t}\pi^*\omega_\Sigma=\Tr_{\omega_t}\omega_
\infty\leq C,    
\end{equation}
uniformly for $t\in [0, \infty)$.

Lemma \ref{ddt_phi:bbd1} tells us that the volume form of 
$\omega_t$ behaves exactly as predicted by the cohomology 
information. Since it is useful for establishing the main theorem, we summarize it 
in the following lemma:
\begin{lemma}
\label{ddt_phi:bbd2}
There exists a uniform constant $C = C(n, r, \omega_0, \omega_
\Sigma) > 0$ such that for any $t \in [0, \infty)$, we have
\begin{equation}
\label{volume_form_bbd}
C^{-1}e^{-rt}\Omega\leq \omega^n_t\leq Ce^{-rt}
\Omega.
\end{equation}
\end{lemma}

We now show the K\"ahler potential $\varphi_t$ decays at 
a rate of $e^{-t}$ after a suitable normalization described 
below.

For each $z \in \Sigma$ and $t \in [0,T)$, we denote $\omega_
{t,z}$ to be the restriction of $\omega_t$ on the fibre $\pi^
{-1}(z)$. For each $t \in [0,T)$, we define a function $\Phi_t: 
\Sigma \to \R$ by 
\[\Phi_t (z) = \frac{1}{\Vol_{\omega_{0,z}}(\pi^{-1}(z))}
\int_{\pi^{-1}(z)} \varphi_t ~ \omega_{0,z}^r\]
which is the average value of $\varphi_t$ over each fibre 
$\pi^{-1}(z)$. The pull-back $\pi^*\Phi_t$ is then a function defined on $X$. 
For simplicity, we also denote $\pi^*\Phi_t$ by $\Phi_t$.

\begin{lemma}\label{potential_bbd}
There exists a uniform constant $C = C(n, r, \omega_0, 
\omega_\Sigma)$ such that for any $t \in [0, \infty)$, we 
have
\begin{equation}
\abs{e^t(\varphi_t - \Phi_t)} \leq C. 
\end{equation}
\end{lemma}

\begin{proof}

Denote $\tilde\varphi_t=e^t(\varphi_t - \Phi_t)$. For each 
$z \in \Sigma$, we have $\hat\omega_{t,z} = e^{-t}
\omega_{0,z}$, and so 
\[\omega_{t,z}=e^{-t}\omega_{0,z} + \eval{\sqrt{-1}
\ddbar\varphi_t}{\pi^{-1}(z)}.\]

Since $\Phi_t$ depends only on $z \in \Sigma$, we have $
\eval{\sqrt{-1}\ddbar\Phi_t}{\pi^{-1}(z)} = 0$. By 
rearranging, we have
\begin{equation}
\label{metric_restricted}
e^t\omega_{t,z}=\omega_{0,z} + \eval{\sqrt{-1}\ddbar
\tilde\varphi_t}{\pi^{-1}(z)}.
\end{equation}

Regard \eqref{metric_restricted} to be a metric equation 
on the manifold $\pi^{-1}(z)$, and we have
\begin{equation}
\label{volume_r}
\left(\omega_{0,z} + \eval{\sqrt{-1}\ddbar\tilde\varphi_t}
{\pi^{-1}(z)}\right)^r=\left(e^t\omega_{t,z}\right)^r
\end{equation}

Using Lemma \ref{ddt_phi:bbd2}, we can see along $\pi^
{-1}(z)$,
\begin{align}\label{volume_r_bbd}
\frac{\omega_{t,z}^r}{\omega_{0,z}^r} 
&= \frac{\omega_t^r \wedge (\pi^*\omega_\Sigma)^{n-r}}
{\omega_0^r \wedge (\pi^*\omega_\Sigma)^{n-r}} \\
\nonumber 
&= \frac{\omega_t^r \wedge (\pi^*\omega_\Sigma)^{n-r}}
{\omega_t^n} \cdot \frac{\omega_t^n}{\omega_0^r \wedge 
(\pi^*\omega_\Sigma)^{n-r}}\\
\nonumber 
& \leq C(\Tr_{\omega_t}\pi^*\omega_\Sigma)^{n-r} \cdot e^{-rt}.
\end{align}

Combining \eqref{schwarz} with \eqref{volume_r_bbd}, we 
see that \eqref{volume_r} can be restated as
\begin{equation}
\left(\omega_{0,z} + \eval{\sqrt{-1}\ddbar\tilde\varphi_t}
{\pi^{-1}(z)}\right)^r = F_z(\xi,t) \left(\omega_{0,z}\right)^r
\end{equation}
where $F_z(\xi,t) : \pi^{-1}(z) \times [0,T) \to \R_{>0}$ is 
uniformly bounded from above.

Since $\int_{\pi^{-1}(z)} \tilde\varphi_t \omega_{0,z}^r = 0$, 
by applying Yau's $L^\infty$-estimate (see \cite{Y78}) on \eqref
{metric_restricted}, we then have
\begin{equation}
\sup_{\pi^{-1}(z) \times [0,T)}|\tilde\varphi_t| \leq C_z,
\end{equation}
where $C_z$ depends on $n, r, \omega_0, \omega_\Sigma, 
\sup_{\pi^{-1}(z) \times [0,T)}F_z, \Vol_{\omega_{0,z}}
(\pi^{-1}(z))$, the Sobolev and Poincar\'e constants of 
$\pi^{-1}(z)$ with respect to metric $\omega_{0,z}$, all 
of which can be bounded uniformly independent of $z$. 
It completes the proof of the lemma.
\end{proof}

\begin{remark}
Yau's $L^{\infty}$-estimate was proved by a Moser's iteration argument. Readers may refer to Chapter 2 of \cite{Siu87} for an exposition of the proof.
\end{remark}
\begin{remark}
In our setting, the uniform boundedness of Sobolev and Poincar\'e 
constants of $(\pi^{-1}(z), \omega_{0,z})$ follows from the 
compactness of $\Sigma$ and the absence of singular fibres. 
With the presence of singular fibres, there is a detail discussion in \cite{To10} in this regard. The bounds of these constants can be derived using the fact that $\pi^{-1}(z)$'s are minimal submanifolds of $X$ and the classical results in \cite{MS73, Ch75, LY80}.
\end{remark}

\section{Proof of Theorem \ref{main}}
Now we can proceed to the proof of the main result about the collapsing rate.
\begin{proof}
[Proof of Theorem \ref{main}]
We apply Maximum Principle to the following quantity
\[Q := \log(e^{-t}\Tr_{\omega_t}\omega_0)-Ae^t(\varphi_t - 
\Phi_t),\]
where $A$ is a positive constant to be chosen. Denote 
$\square = \partial_t - \Delta$, and we have
\begin{align}
\label{metric_lower_bd:ineq1}
\square\log(e^{-t}\Tr_{\omega_t}\omega_0) 
& \leq C+C\Tr_{\omega_t}\omega_0
\end{align}
where $C$ depends on the curvature of $\omega_0$.

We also need to compute the evolution equation for the second term in $Q$. 
\begin{align*}
\square Ae^t(\varphi_t - \Phi_t) 
&= Ae^t\left(\D{\varphi_t}{t} - \D{\Phi_t}{t}
\right) +Ae^t(\varphi_t - \Phi_t)\\
& \quad -Ae^t(\Delta\varphi_t - \Delta \Phi_t)\\
&\geq Ae^t\left(\frac{\p \varphi_t}{\p t} - \int_
{\pi^{-1}(z)} \D{\varphi_t}{t} \omega_{0,z}^r 
\right) - CA\\
&\quad - Ae^t(n-\Tr_{\omega_t}\hat\omega_t - 
\Delta\Phi_t).
\end{align*}

Using the lower bound of $\frac{\p \varphi_t}{\p t}$ given by Lemma \ref{ddt_phi:bbd1}, 
we have
\begin{align}\label{metric_lower_bd:ineq2}
\square Ae^t(\varphi_t - \Phi_t) & \geq -CAe^t+
Ae^t\Tr_{\omega_t}\hat\omega_t\\
\nonumber & \quad + Ae^t \left(\Delta\Phi_t-\int_
{\pi^{-1}(z)} \D{\varphi_t}{t} \omega_{0,z}^r\right).
\end{align}

Combining \eqref{metric_lower_bd:ineq1} and \eqref
{metric_lower_bd:ineq2}, we have
\begin{align}
\label{metric_lower_bd:ineq3}
\square Q & \leq CAe^t + C\Tr_{\omega_t}\omega_0 - 
Ae^t\Tr_{\omega_t}\left(e^{-t}\omega_0+(1-e^{-t})
\omega_\infty\right)\\
\nonumber 
& \quad -Ae^t \left(\Delta\Phi_t-\int_{\pi^{-1}(z)} \D
{\varphi_t}{t} \omega_{0,z}^r\right)\\
\nonumber 
& \leq CAe^t + (C - A)\Tr_{\omega_t}\omega_0\\
\nonumber 
& \quad - Ae^t \left(\Delta\Phi_t-\int_{\pi^{-1}(z)} \D
{\varphi_t}{t} \omega_{0,z}^r\right).
\end{align}

By Lemma \ref{ddt_phi:bbd1}, we have $\D{\varphi_t}
{t} \leq C$ for some uniform constant $C$. It follows that
\[\int_{\pi^{-1}(z)} \D{\varphi_t}{t} \omega_{0,z}^r 
\leq C.\]

Note that $\Vol_{\omega_{0,z}} (\pi^{-1}(z))$ is actually 
independent of $z$.

For the Laplacian term of $\Phi_t$, we have
\begin{align*}
\Delta \int_{\pi^{-1}(z)}\varphi_t \omega_{0,z}^r 
&= \Tr_{\omega_t}\int_{\pi^{-1}(z)} \sqrt{-1}\ddbar\varphi_t 
\wedge \omega_{0,z}^r\\
&= \Tr_{\omega_t}\int_{\pi^{-1}(z)} (\omega_t-\hat\omega_t) 
\wedge \omega_{0,z}^r\\
&\geq -\Tr_{\omega_t} \int_{\pi^{-1}(z)} \hat\omega_t 
\wedge \omega_{0,z}^r\\
&\geq -\Tr_{\omega_t}\int_{\pi^{-1}(z)}\left(\omega_0
\wedge \omega_{0,z}^r + \pi^*\omega_\Sigma \wedge 
\omega_{0,z}^r\right).
\end{align*}

Since $\Tr_{\omega_t}\pi^*\omega_\Sigma \leq C$ and 
$\int_{\pi^{-1}(z)}\left(\omega_0 \wedge \omega_{0,z}^r + 
\pi^*\omega_\Sigma \wedge \omega_{0,z}^r\right)$ is a 
smooth $(1,1)$-form on $\Sigma$ independent of $t$, we have 
\[\Delta\int_{\pi^{-1}(z)}\varphi_t \omega_{0,z}^r \geq 
-C\]
for some uniform constant $C$. Back to \eqref{metric_lower_bd:ineq3}, we have
\begin{equation}
\square Q \leq CAe^t+(C-A)\Tr_{\omega_t}\omega_0 
\leq CAe^t - \Tr_{\omega_t}\omega_0
\end{equation}
if we choose $A$ sufficiently large such that $C-A\leq -1$.

Hence, for any $S> 0$, at the point where $Q$ achieves its 
maximum over $X \times [0, S]$, we have $\Tr_{\omega_t}
(e^{-t}\omega_0)\leq C$ for some uniform constant 
$C$ independent of $S$. Together with Lemma \ref{potential_bbd}, 
it follows that for any $t \in [0, \infty)$ we have,
\begin{equation}
\label{metric_lower_bbd:semifinial}
C^{-1}e^{-t}\omega_0 \leq \omega_t.
\end{equation}
Combining with the fact from \eqref{schwarz} that $\omega_t \geq C^{-1}\pi^*
\omega_\Sigma$, we have
\begin{equation}
\label{metric_lower_bbd:final}
C^{-1}\hat\omega_t \leq \omega_t.
\end{equation}
Together with  Lemma \ref{ddt_phi:bbd2} which indicates $\omega^
n_t\leq C\hat\omega^n_t$, we also have $\omega_t \leq 
C\hat\omega_t$ for any $t \in [0, \infty)$.

It completes the proof of the theorem. 
\end{proof}

\section{Convergence at Time Infinity}

The argument in \cite{ST07} can be directly applied to our regular 
infinite time singularity case for general dimension and show that 
the K\"ahler-Ricci flow converges to the commonly called generalized 
K\"ahler-Einstein metric. This is done in more general setting in 
\cite{ST08}, and we include it here for completeness.

We focus on the non-trivial case $\dim_\mathbb{C}\Sigma
\geq 1$. The fibres of the map $\pi: X\to \Sigma$ are all 
smooth Calabi-Yau manifolds, and so there is a Ricci-flat metric 
$\omega_{0, z}+\sqrt{-1}\p\bar\p \Psi(z)$ for each $z\in\Sigma$. 
After normalizing $\Psi(z)$ to have $\int_{\pi^{-1}(z)}\Psi(z)
\omega^r_{0, z}=0$, we have a smooth function $\Psi$ over 
$X$ with the smooth closed $(1, 1)$-form 
$$\omega_{SF}=\omega_0+\sqrt{-1}\p\bar\p\Psi$$
being Ricci flat on each fibre. We further define the 
following smooth function \textit{a priori} on $X$,  
$$F=\frac{\Omega}{(^n_r)\omega^{n-r}_\infty\wedge 
\omega^r_{SF}}$$
which makes sense despite of the fact that $\omega_{SF}$ might 
not be a metric over $X$. 

Since $\sqrt{-1}\p\bar\p \log\Omega=\omega_\infty=\pi^*\omega_
\Sigma$ and $\omega_{SF}$ is a Ricci-flat metric along each fibre, 
we know that $F$ is constant along each fibre and so is the 
pull-back of a smooth function over $\Sigma$. 

Over $\Sigma$, we always have a unique and smooth solution $u$ to the following 
complex Monge-Amp\`ere equation, which is a classic elliptic equation 
when $\dim_\mathbb{C}\Sigma=1$, 
$$(\omega_\Sigma+\sqrt{-1}\p\bar\p u)^{n-r}=Fe^u\omega^
{n-r}_\Sigma,$$
and we use the same notation for its pull-back on $X$.

Denote $\omega_{GKE}=\omega_\Sigma+\sqrt{-1}\p\bar\p u$. Direct 
computation as in \cite{ST07} then shows that this metric satisfies 
$$\Ric(\omega_{GKE})=-\omega_{GKE}+\omega_{WP}$$
where $\omega_{WP}$ is the Weil-Petersson metric determined by 
the fibration $\pi: X\to\Sigma$. We also use $\omega_{GKE}$ 
for its pull-back on $X$, and so on $X$, $\omega_{GKE}=
\omega_\infty+\sqrt{-1}\p\bar\p u$.

Our main result in this section is the following. 
\begin{theorem}\label{thm:convergence}
\label{potential-con}
The solution $\varphi_t$ for (\ref{MAP}) converges uniformly to 
$u$ as $t\to\infty$.
\end{theorem}

\begin{proof}

Set $v_t=\varphi_t-u-e^{-t}\Psi$. Since the flow metric $\omega_t=
\hat\omega_t+\sqrt{-1}\p\bar\p \varphi_t$ and $\hat\omega_t=
\omega_\infty+e^{-t}(\omega_0-\omega_\infty)$, we have 
$$\omega_t=(\omega_{GKE}-e^{-t}\omega_\infty)+e^{-t}
\omega_{SF}+\sqrt{-1}\p\bar\p v_t.$$

Meanwhile, since $\omega^{n-r}_{GKE}=Fe^u\omega^{n-r}_
\infty$ and $F=\frac{\Omega}{(^n_r)\omega^{n-r}_\infty\wedge 
\omega^r_{SF}}$, we have 
$$(^n_r)\omega^{n-r}_{GKE}\wedge\omega^r_{SF}=\Omega 
e^u.$$

Combine this to compute the evolution of $v$ as follows
\begin{equation}
\begin{split}
\frac{\p v_t}{\p t}
&= \frac{\p \varphi_t}{\p t}+e^{-t}\Psi \\
&= \log\frac{e^{rt}\((\omega_{GKE}-e^{-t}\omega_\infty)+e^{-t}
\omega_{SF}+\sqrt{-1}\p\bar\p v_t\)^n}{\Omega}-\varphi+e^{-t}\Psi \\
&= \log \frac{e^{rt}\((\omega_{GKE}-e^{-t}\omega_\infty)+e^{-t}
\omega_{SF}+\sqrt{-1}\p\bar\p v_t\)^n}{(^n_r)\omega^{n-r}_{GKE}
\wedge\omega^r_{SF}}+u-\varphi+e^{-t}\Psi \\
&=  \log \frac{e^{rt}\((\omega_{GKE}-e^{-t}\omega_\infty)+e^{-t}
\omega_{SF}+\sqrt{-1}\p\bar\p v_t\)^n}{(^n_r)\omega^{n-r}_{GKE}
\wedge\omega^r_{SF}}-v_t.
\end{split}
\end{equation}

Now we apply the standard Maximum Principle argument for $v_t$ by 
looking at the spatial extremal value as a function of $t$. 

The following observation is very useful 
$$-Ce^{-t}\leq \log \frac{e^{rt}\((\omega_{GKE}-e^{-t}\omega_
\infty)+e^{-t}\omega_{SF}\)^n}{(^n_r)\omega^{n-r}_{GKE}\wedge
\omega^r_{SF}}\leq Ce^{-t}.$$

Set $A(t)=\max_{X} v_t$ and we have 
$$\frac{d A}{d t}\leq Ce^{-t}-A$$ 
and so $v_t\leq Cte^{-t}+Ce^{-t}$. Similarly $v_t\geq -Cte^{-t}
-Ce^{-t}$.

Hence we conclude $|\varphi_t-u|\leq Ce^{-\frac{t}{2}}$, and $\varphi_t \to u$ exponentially.

\end{proof}

Recall that Theorem \ref{main} proves $\omega_t$ and $\hat{\omega}_t$ are uniformly equivalent. Combining with the fact that $\Tr_{\omega_0}\hat{\omega
}_t \leq C$, one can show $|\Delta_{\omega_0}\varphi_t| \leq C$ and hence we have
\begin{corollary}\label{cor:C_one_alpha}
\label{metric-con}
K\"ahler-Ricci flow $\omega_t$ converges to $\omega_{GKE}$ as $t\to
\infty$ in the sense that the metric potential $\varphi_t\to u$ in $C^
{1, \alpha}$-norm for any $\alpha<1$.
\end{corollary}

\section{Type III singularity of Toric Fibrations}
In this section, we specialize on one category of Calabi-Yau fibrations, namely \textit{toric fibrations}, where all fibres $\pi^{-1}(z)$ are complex tori $\C^r / \Lambda_z$. We again focus on regular fibrations. We will provide another geometric application to the collapsing rate result (Theorem \ref{main}), obtaining the uniform boundedness of $\|\Rm\|_{\omega_t}$ when the initial K\"ahler class $[\omega_0]$ is rational. A solution $\tilde{\omega}_s$ to the unnormalized Ricci flow $\partial_s \tilde{\omega}_s = -\Ric(\tilde{\omega}_s)$ is called \textbf{Type III} if $\|\Rm\|_{\tilde{\omega}_s} \leq C/s$ for some uniform constant $C > 0$. One can easily verify by the correspondence $\omega_t = e^{-t}\tilde{\omega}_{(e^t -1)}$ between normalized and unnormalized flows that Type III singularity is equivalent to saying $\|\Rm\|_{\omega_t}$ is uniformly bounded in the normalized flow. Furthermore, we will show that in this special case the convergence of both $\varphi_t$ and $\omega_t$ is in fact in $C^\infty$-topology which strengthened the result showed in Corollary \ref{cor:C_one_alpha}. 

Here is the setting in this section. Let $X^n \xrightarrow{\pi} \Sigma^{n-r}$ be a holomorphic submersion fibred by complex tori such that $c_1(X) = -[\pi^*\omega_\Sigma]$ for some K\"ahler metric $\omega_\Sigma$ on $\Sigma$. For each point $z \in \Sigma$, there exists a neighborhood $z \in B \subset \Sigma$ such that $\pi^{-1}(B) \subset X$ is trivialized: i.e. there exists a lattice section $\Lambda_z$ varying over $z \in B$ such that $(B \times \C^r) / \Lambda_z$ is biholomorphic to $\pi^{-1}(B)$.

From now on we assume the initial K\"ahler class $[\omega_0]$ is rational, i.e. $[\omega_0] \in H^2(X, \Q)$, and hence $X$ must be projective. Then there exists a closed nonnegative semi-flat form $\omega_{SF}$ on $\pi^{-1}(B)$, with a good rescaling property, such that on each fibre $\pi^{-1}(z)$ we have $\eval{\omega_{SF}}{\pi^{-1}(z)}$ cohomologous to $\eval{\omega_0}{\pi^{-1}(z)}$. The semi-flat form $\omega_{SF}$ is a $(1,1)$-form such that for each $z \in B$ the restriction $\eval{\omega_{SF}}{\pi^{-1}(z)}$ on the fibre $\pi^{-1}(z)$ is flat.

\begin{lemma}[Gross-Tosatti-Y.Zhang \cite{GTZ11}]\label{lma:semiflat}
Given that $X$ is projective and $[\omega_0]$ is rational, then one can find a closed nonnegative $(1,1)$-form $\omega_{SF}$ such that there exists a smooth function $f : \pi^{-1}(B) \to \R$ with
\begin{equation*}
\omega_{SF} - \omega_0 = \sqrt{-1}\ddbar f
\end{equation*}
and passing to the universal cover $p: B \times \C^r \to B \times (\C^r / \Lambda_z)$, we have
\begin{equation*}
p^*\omega_{SF} = \sqrt{-1}\ddbar\psi
\end{equation*}
where $\psi : B \times \C^r \to \R$ is a smooth function with the following rescaling property:
$$\psi(z, \lambda\xi) = \lambda^2\psi(z, \xi) \quad \text{for any } (z, \xi) \in B \times \C^r \text{ and } \lambda \in \R.$$
\end{lemma}
Denote $\lambda_t : B \times \C^r \to B \times \C^r$ to be the rescaling map $(z, \xi) \mapsto (z, e^{t/2}\xi)$. One can easily verify that
\begin{equation}\label{eq:semiflat_rescaling}
e^{-t}\lambda_t^*p^*\omega_{SF} = e^{-t}\lambda_t^*\sqrt{-1}\ddbar\psi = e^{-t}\sqrt{-1}\ddbar(\psi \circ \lambda_t) = \sqrt{-1}\ddbar\psi = p^*\omega_{SF}.
\end{equation}

As before, we rewrite the normalized K\"ahler-Ricci flow $\D{\omega_t}{t} = -\Ric(\omega_t) - \omega_t$ as the following complex Monge-Amp\`ere equation \eqref{MAP}
\begin{equation*}
\D{\varphi_t}{t} = \log\frac{(\hat\omega_t + \sqrt{-1}\ddbar\varphi_t)^n}{e^{-rt}\Omega} - \varphi_t,
\end{equation*}
where $\hat\omega_t = e^{-t}\omega_0 + (1-e^{-t})\pi^*\omega_\Sigma$ and $\omega_t = \hat\omega_t + \sqrt{-1}\ddbar\varphi_t$. Here $\Omega$ is a volume form on $X$ such that $\sqrt{-1}\ddbar\log\Omega = \pi^*\omega_\Sigma$. We first establish the following lemma using Theorem \ref{main}:
\begin{lemma}\label{lma:metric_bd_rescaled}
There is a constant $C > 0$ such that on $B \times \C^r$ we have
\begin{equation*}
C^{-1}p^*(\pi^*\omega_\Sigma + \omega_{SF}) \leq \lambda_t^*p^*\omega_t \leq Cp^*(\pi^*\omega_\Sigma + \omega_{SF}) ~~~ \text{ for any } t \geq 1.
\end{equation*}
\end{lemma}
\begin{proof}
First we use the metric equivalence of $\omega_t$ and $\hat\omega_t$ established in Theorem \ref{main}:
$$C^{-1}\hat\omega_t \leq \omega_t \leq C\hat\omega_t.$$
For the sake of simplicity, we denote $\omega_t \simeq \hat\omega_t$ for the above metric equivalence (and for any other pair of metrics). Then,
\begin{align*}
\lambda_t^* p^*\omega_t & \simeq \lambda_t^* p^*\hat\omega_t\\
& = \lambda_t^*p^*(e^{-t}\omega_0 + (1-e^{-t})\pi^*\omega_\Sigma)\\
& = e^{-t}\lambda_t^*p^*\omega_0 + (1-e^{-t})p^*\pi^*\omega_\Sigma.
\end{align*}
Note that $\lambda_t^*p^*\pi^*\omega_\Sigma = p^*\pi^*\omega_\Sigma$ since $\lambda_t$ rescales the fibre directions only. As $\omega_0 \simeq \omega_{SF} + \pi^*\omega_\Sigma$, we have
\begin{align*}
\lambda_t^* p^*\omega_t & \simeq e^{-t}\lambda_t^*p^*\omega_{SF} + (1-e^{-t})p^*\pi^*\omega_\Sigma\\
& \simeq p^*\omega_{SF} + (1-e^{-t})p^*\pi^*\omega_\Sigma\\
& \simeq p^*(\omega_{SF} + \pi^*\omega_\Sigma) ~~~~ \text{ for } t \geq 1.
\end{align*}
\end{proof}
Next we show $\lambda_t^*p^*\omega_t$ is locally cohomologous to $p^*(\omega_\Sigma + \omega_{SF})$ in $B \times \C^r$:
\begin{align*}
\lambda_t^*p^*\omega_t & = \lambda_t^*p^*(e^{-t}\omega_0 + (1-e^{-t})\pi^*\omega_\Sigma) + \sqrt{-1}\ddbar(\varphi_t \circ p \circ \lambda_t)\\
& = e^{-t}\lambda_t^*p^*\omega_0 + (1-e^{-t})p^*\pi^*\omega_\Sigma + \sqrt{-1}\ddbar(\varphi_t \circ p \circ \lambda_t)\\
& = e^{-t}\lambda_t^*p^*(\omega_{SF} - \sqrt{-1}\ddbar f) + (1-e^{-t})p^*\pi^*\omega_\Sigma + \sqrt{-1}\ddbar(\varphi_t \circ p \circ \lambda_t)\\
& = p^*\omega_{SF} - \sqrt{-1}\ddbar(e^{-t} f \circ p \circ \lambda_t) + (1-e^{-t})p^*\pi^*\omega_\Sigma + \sqrt{-1}\ddbar(\varphi_t \circ p \circ \lambda_t).
\end{align*}
On the open ball $B \subset \Sigma$, the K\"ahler metric $\omega_\Sigma$ can be locally expressed as $\sqrt{-1}\ddbar\zeta$ for some smooth function $\zeta : B \to \R$. Therefore, we have
\begin{equation}\label{eq:semiflat}
\lambda_t^*p^*\omega_t = p^*(\omega_{SF} + \pi^*\omega_\Sigma) + \sqrt{-1}\ddbar u_t,
\end{equation}
where $u_t = \varphi_t \circ p \circ \lambda_t - e^{-t} (f \circ p \circ \lambda_t) - e^{-t} (\zeta \circ p)$. As $\varphi_t$ is uniformly bounded on $X$, we have $u_t$ being uniformly bounded on $B \times \C^r$.

One can show the following higher-order estimates using Evans-Krylov's and Schauder's estimates:
\begin{lemma}\label{lma:metric_rescaled:higher_order}
Given any compact set $K \subset B \times \C^r$ and any $k \geq 0$, there exists a constant $C = C(K, k)$ such that
\begin{equation}
\|\lambda_t^*p^*\omega_t\|_{C^k(K, \delta)} \leq C
\end{equation}
where $\delta$ is the Euclidean metric of $B \times \C^r$.
\end{lemma}
\begin{proof}
We first derive a complex Monge-Amp\`ere equation for $u_t$: from the K\"ahler-Ricci flow equation, we have
$$\omega_t^n = e^{\D{\varphi_t}{t} + \varphi_t - rt}\Omega.$$
By rescaling, we have
$$(\lambda_t^*p^*\omega_t)^n = \lambda_t^*p^*\omega_t^n = (e^{\D{\varphi_t}{t} + \varphi_t - rt} \circ p \circ \lambda_t) \cdot \lambda_t^*p^*\Omega.$$

Since $\sqrt{-1}\ddbar\log\Omega = \pi^*\omega_\Sigma$ and so $\eval{\sqrt{-1}\ddbar\log\Omega}{\pi^{-1}(z)} = 0$ for each $z \in B$. By the compactness of the toric fibres $\pi^{-1}(z)$, we have $\Omega$ depends only on $z \in \Sigma$ and hence $e^{-rt}\lambda_t^*p^*\Omega = p^*\Omega$. Therefore, from \eqref{eq:semiflat} the potential $u_t$ satisfies the following equation:
\begin{equation}\label{eq:evans_krylov}
\log(p^*(\omega_{SF} + \pi^*\omega_\Sigma) + \sqrt{-1}\ddbar u_t)^n = \left(\D{\varphi_t}{t} + \varphi_t\right) \circ p \circ \lambda_t  + \log(p^*\Omega).
\end{equation}
The following quantities are uniformly bounded according to the gradient and Laplacian estimates due to \cite{ChY75, LY80} (see also \cite{SeT08, Z09, ST07, ST08} etc.)
$$\|\nabla_{\omega_t}(\dot\varphi_t + \varphi_t)\|_{\omega_t} \leq C, ~~~ |\Delta(\dot\varphi_t + \varphi_t)| \leq C.$$
Hence, 
\begin{align*}
\|\nabla_{\lambda_t^*p^*\omega_t}(\dot\varphi_t + \varphi_t) \circ p \circ \lambda_t\|_{\lambda_t^*p^*\omega_t} & \leq C,\\
|\Delta_{\lambda_t^*p^*\omega_t}(\dot\varphi_t + \varphi_t) \circ p \circ \lambda_t| & \leq C.
\end{align*}
By Lemma \ref{lma:metric_bd_rescaled}, we have $\lambda_t^*p^*\omega_t \simeq p^*(\omega_{SF} + \pi^*\omega_\Sigma) \simeq \delta$ on $K \subset B \times \C^r$. Hence, apply Evans-Krylov's theory (see \cite{Ev82, Kr83}) on \eqref{eq:evans_krylov} one can get a uniform $C^{2, \alpha}$-estimate on $u_t$. Finally, by the Schauder's estimate (see e.g. \cite{GT}) and a bootstrapping argument, one can complete the proof of the lemma. Here we supply the detail of the bootstrapping argument:

Let $D$ be any first-order differential operator on $B \times \C^r$. Differentiating \eqref{eq:evans_krylov} by $D$ gives
\begin{align}\label{eq:bootstrap:1}
\nonumber \Delta_{\lambda^*_tp^*\omega_t}(Du_t) & = - \Tr_{\lambda^*_tp^*\omega_t} Dp^*(\omega_{SF}+\pi^*\omega_\Sigma)\\
& \quad + D\left\{\lambda^*_tp^*\left(\D{\varphi_t}{t}+\varphi_t\right) \right\} + D \log(p^*\Omega).
\end{align}
From \eqref{eq:box_dphi+phi}, one can show using the chain rule that
\begin{align*}
\dd{t} \left\{\lambda^*_tp^*\left(\D{\varphi_t}{t} + \varphi_t\right)\right\} & = \lambda_t^*p^*\left\{\Delta_{\omega_t}\left(\D{\varphi_t}{t} + \varphi_t\right) - n + r + \Tr_{\omega_t}\pi^*\omega_\Sigma\right\} \\
& \quad + \sum_{j=1}^r \lambda^*_t p^* \dd{\xi_j} \left(\D{\varphi_t}{t} + \varphi_t\right) \cdot \frac{1}{2}e^{\frac{t}{2}} \xi_j\\
& \quad + \sum_{j=1}^r \lambda^*_t p^* \dd{\xi_{\bar{j}}} \left(\D{\varphi_t}{t} + \varphi_t\right) \cdot \frac{1}{2}e^{\frac{t}{2}} \bar{\xi}_j\\
& = \Delta_{\lambda^*_tp^*\omega_t} \lambda^*_tp^* \left(\D{\varphi_t}{t} + \varphi_t\right) - n + r + \Tr_{\lambda^*_tp^*\omega_t}p^*\pi^*\omega_\Sigma\\
& \quad + \frac{1}{2} \sum_{j=1}^r \dd{\xi_j}\left\{\lambda^*_tp^*\left(\D{\varphi_t}{t} + \varphi_t\right)\right\} \cdot \xi_j\\
& \quad + \frac{1}{2} \sum_{j=1}^r \dd{\xi_{\bar{j}}}\left\{\lambda^*_tp^*\left(\D{\varphi_t}{t} + \varphi_t\right)\right\} \cdot \bar{\xi}_j
\end{align*}

Hence, $\displaystyle{\lambda^*_tp^*\left(\D{\varphi_t}{t} + \varphi_t\right)}$ satisfies the following parabolic equation:
\begin{equation}\label{eq:bootstrap:2}
\left(\dd{t} - \Delta_{\lambda^*_tp^*\omega_t}\right) H = \langle \partial_{\mathbb{E}} H, \partial_\xi + \bar{\partial}_\xi \rangle_\delta - n +r +\Tr_{\lambda^*_tp^*\omega_t}p^*\pi^*\omega_\Sigma
\end{equation}
where $\partial_\mathbb{E}$ denotes the flat connection on $B \times \C^r$ and $\partial_\xi = (\xi_1, \ldots, \xi_r) \in \C^r$.

Assume that $u_t \in C^{k, \alpha}$ for some $k \geq 2$ and $0 < \alpha < 1$. Then by \eqref{eq:semiflat} we have $\lambda_t^*p^*\omega_t \in C^{k-2, \alpha}$. By the uniform bound of $\lambda_t^*p^*\omega_t$, one also has $(\lambda_t^*p^*\omega_t)^{-1} \in C^{k-2, \alpha}$. Hence applying parabolic Schauder's estimate on \eqref{eq:bootstrap:2} one get
$$\lambda^*_tp^*\left(\D{\varphi_t}{t} + \varphi_t\right) \in C^{k,\alpha}$$
The controls are uniform in time because we already have uniform controls on the metric and $C^0$ norm of the evolution term.

Hence the coefficients of the elliptic equation \eqref{eq:bootstrap:1} are in $C^{k-2, \alpha}$, and applying elliptic Schauder's estimate one has
$Du_t \in C^{k, \alpha}$ and therefore $u_t \in C^{k+1, \alpha}$ which is one higher-order up than our assumption. Since Evans-Krylov's theory asserts that $u_t \in C^{2,\alpha}$, this bootstrapping argument implies $u_t \in C^\infty$ which completes the proof of the lemma.
\end{proof}

Lemma \ref{lma:metric_rescaled:higher_order} proves smooth convergence of the modified potential $u_t$. The uniform bound on $\|\Rm\|_{\omega_t}$ can hence be established by the following argument.
\begin{remark}
In fact, a uniform bound for $C^4$-norm of $u_t$ is sufficient to prove the uniform boundedness of $\|\Rm\|_{\omega_t}$. The higher order estimates will be used in obtaining later results.
\end{remark}
For each point $x \in X$, find a compact subset $K$ containing $x$ such that $K \subset \pi^{-1}(B) \equiv B \times (\C^r / \Lambda_z)$ for some small open ball $B \subset \Sigma$. We then get
$$\sup_K \|\Rm\|_{\omega_t} = \sup_{K'}\|\Rm\|_{p^*\omega_t}$$
for some $K' \subset B \times \C^r$ such that $p(K')=K$. Therefore,
$$\sup_K \|\Rm\|_{\omega_t} = \sup_{\lambda_t^{-1}(K')}\|\Rm\|_{\lambda_t^*p^*\omega_t}.$$
As $\lambda_t^{-1}(K') = \{(z,e^{-t/2}\xi) : (z, \xi) \in K'.\}$, one can easily see $\cup_{t>0}\lambda_t^{-1}(K')$ is precompact. By Lemma \ref{lma:metric_rescaled:higher_order}, one has $\sup_{\lambda^{-1}(K')}\|\Rm\|_{\lambda_t^*p^*\omega_t} \leq C_K$ where $C_K$ depends on $K$. By covering the compact manifold $X$ by finitely many such $K$'s we have proved:
\begin{proposition}\label{prop:type_III}
Suppose $\pi:X\to\Sigma$ is a smooth holomorphic submersion fibreed by complex tori such that the initial K\"ahler class $[\omega_0]$ is rational. Then along the normalized K\"ahler-Ricci flow \eqref{KRF}, we have $\|\Rm\|_{\omega_t} \leq C$ for some constant $C > 0$ independent of $t$, i.e. the flow encounters Type III singularity.
\end{proposition}

Another consequence of Lemma \ref{lma:metric_rescaled:higher_order} is the $C^\infty$-convergence of $\omega_t$ to the generalized K\"ahler-Einstein metric, which strengthened the $C^{1,\alpha}$-convergence result (on the potential level) in Corollary \ref{cor:C_one_alpha}. Recall that $\varphi_t \to u$ as $t \to \infty$ where $u : \Sigma \to \R$ is the potential function such that $\omega_{GKE} = \omega_\Sigma + \sqrt{-1}\ddbar u$. Under the setting in this section, we have the following proposition:

\begin{proposition}\label{prop:C_infty}
Under the same assumption as in Proposition \ref{prop:type_III}, we have
\begin{enumerate}[(i)]
\item $\varphi_t \to u$ in $C^{\infty}(X, \omega_0)$-topology, and
\item $\omega_t \to \pi^*\omega_{GKE}$ in $C^\infty(X, \omega_0)$-topology.
\end{enumerate}
\end{proposition}
\begin{remark}
From now on all the $C^k$-norms below are with respect to a time-independent metric. Also, by uniform bounds on $C^k$-norms we mean that the bounds are independent of $t$ but may depend on $k$.
\end{remark}
\begin{proof}
First fix a compact set $K \subset M$ and find $K' \subset B \times \C^r$ such that $K'$ and $K$ are biholomorphic via $p$, i.e. $p(K') = K$. From Theorem \ref{thm:convergence} we already know that $\varphi_t \to u$ in $C^0$-norm, hence to prove (i) it suffices to establish uniform bounds on $\|\varphi_t\|_{C^k(K)}$. Note that
$$p^*\omega_t = p^*\hat{\omega}_t + \sqrt{-1}\ddbar(\varphi_t \circ p)$$
and it is straight-forward to check that $\|p^*\hat{\omega}_t\|_{C^k(K')} \leq C(K',k)$ for some constant $C > 0$ depending only on $K'$ and $k$. We are left to show $\|p^*\omega_t\|_{C^k(K')}$ is uniformly bounded independent of $t$.

Denote $\{z_i, \xi_\alpha\}$ to be the base-fibre coordinates on $B \times \C^r$, i.e. $i = 1, \ldots, n-r$ and $\alpha = 1, \ldots, r$. The local components of $p^*\omega_t$ and $\lambda^*_tp^*\omega_t$ are related by
\begin{align*}
(p^*\omega_t)_{i\bar j}(z, \xi) & = (\lambda^*_tp^*\omega_t)_{i\bar j}(z, e^{-t/2}\xi),\\
(p^*\omega_t)_{i\bar\alpha}(z, \xi) & = e^{-t/2} (\lambda^*_tp^*\omega_t)_{i\bar{\alpha}}(z, e^{-t/2}\xi),\\
(p^*\omega_t)_{\beta\bar{j}}(z,\xi) & = e^{-t/2} (\lambda^*_tp^*\omega_t)_{\beta\bar{j}}(z, e^{-t/2}\xi),\\
(p^*\omega_t)_{\alpha\bar{\beta}}(z, \xi) & = e^{-t}(\lambda^*_tp^*\omega_t)_{\alpha\bar{\beta}}(z, e^{-t/2}\xi).
\end{align*}
By Lemma \ref{lma:metric_rescaled:higher_order}, the local components of $\lambda^*_tp^*\omega_t$ are uniformly bounded in every $C^k$-norm. It is easy to check from the above relations that the local components of $p^*\omega_t$ are also uniformly bounded in every $C^k$-norm. Combining with the uniform $C^k$-bounds on $p^*\hat{\omega}_t$, we establish the uniform bounds on $\|p^*\varphi_t\|_{C^k(K')}$ and hence $\|\varphi_t\|_{C^k(K)}$. One can then prove (i) by covering $M$ by finitely many compact subsets $K$.

(ii) is a direct consequence of Theorem \ref{thm:convergence} and (i) above. Now we have $\varphi_t \to \pi^*u$ and $\hat{\omega}_t \to \pi^*\omega_\Sigma$ both in $C^{\infty}$-topology. Hence $\omega_t \to \pi^*\omega_\Sigma + \pi^*\sqrt{-1}\ddbar u = \pi^*\omega_{GKE}$ in $C^\infty$-topology as $t \to \infty$.
\end{proof}

To finish this section, we prove a result concerning fibre-wise convergence. We establish that the flow metric restricted on each fibre converges smoothly, after a suitable rescaling, to a flat metric on the torus fibre. Precisely, we have

\begin{proposition}\label{prop:fiber}
Under the same assumption as in Proposition \ref{prop:type_III}, we have
\begin{enumerate}[(i)]
\item $u_t \to u \circ p$ in $C^{\infty}_{\textup{loc}}(B \times \C^r)$ as $t\to\infty$,
\item $\eval{e^t\omega_t}{\pi^{-1}(z)} \to \eval{\omega_{SF}}{\pi^{-1}(z)}$ in $C^{\infty}(\pi^{-1}(z))$-topology.
\end{enumerate}
\end{proposition}
\begin{proof}
By the proof of Lemma \ref{lma:metric_rescaled:higher_order} we have uniform bounds on $\|u_t\|_{C^k(K)}$ for any compact subset $K \subset B \times \C^r$. Hence for (i) it suffices to show $u_t \to u \circ p$ in $C^0$-norm. Recall that $u_t$ is defined by
$$u_t = \varphi_t \circ p \circ \lambda_t - e^{-t}(f \circ p \circ \lambda_t) - e^{-t}(\zeta \circ p)$$
where $f$ and $\zeta$ are time-independent functions and hence are bounded on any compact subset of $B \times \C^r$. It suffices to show $\varphi_t \circ p \circ \lambda_t \to u \circ p$ in $C^0$-norm, it can be established by Lemma \ref{potential_bbd} and Theorem \ref{thm:convergence} as below:
\begin{align*}
& |\varphi_t \circ p \circ \lambda_t (z, \xi) - u \circ p(z,\xi)|\\
& \leq |\varphi_t(z, e^{t/2}\xi) - \varphi_t(z,\xi)| \circ p + |\varphi_t(z,\xi) - u(z,\xi)| \circ p\\
& = O(e^{-t}) + O(e^{-t/2}) = O(e^{-t/2}).
\end{align*}
Taking $t \to \infty$ completes the proof of (i).

To prove (ii), we restrict \eqref{eq:semiflat} to the fibres,
$$\eval{\lambda_t^*p^*\omega_t}{\{z\}\times\C^r} = \eval{p^*\omega_{SF}}{\{z\}\times\C^r} + \eval{\sqrt{-1}\ddbar u_t}{\{z\}\times\C^r}.$$
Pulling-back by $\lambda_{-t}$ defined by $(z,\xi) \mapsto (z, e^{-t/2}\xi)$ gives
$$\eval{p^*\omega_t}{\{z\}\times\C^r} = \eval{\lambda_{-t}^*p^*\omega_{SF}}{\{z\}\times\C^r} + \lambda_{-t}^*\eval{\sqrt{-1}\ddbar u_t}{\{z\}\times\C^r}.$$
By the rescaling property of $\omega_{SF}$ given by \eqref{eq:semiflat_rescaling}, we have
$$\lambda_{-t}^*p^*\omega_{SF} = e^{-t}p^*\omega_{SF}.$$
Note also that in the Euclidean space $\mathbb{C}^r=\{z\}\times\mathbb{C}^r$, 
$$(\lambda_{-t}^*\eval{\sqrt{-1}\ddbar u_t}{\{z\}\times\C^r})(z, \xi) = e^{-t}(\eval{\sqrt{-1}\ddbar u_t}{\{z\}\times\C^r})(z, e^{-t/2}\xi).$$
Combining these, we have
$$(e^t\eval{p^*\omega_t}{\{z\}\times\C^r})(z, \xi) = (\eval{p^*\omega_{SF}}{\{z\}\times\C^r})(z, \xi) + (\eval{\sqrt{-1}\ddbar u_t}{\{z\}\times\C^r})(z, e^{-t/2}\xi).$$
From (i), we have $u_t \to u \circ p$ in $C^{\infty}_{\textup{loc}}(B\times\C^r)$ and since $u \circ p$ depends only on $z \in B$, we have
$$\eval{\sqrt{-1}\ddbar u_t}{\{z\}\times\C^r} \to 0$$
as $t \to \infty$ in $C^{\infty}_{\textup{loc}}$-topology. Hence, we have $e^t\eval{p^*\omega_t}{\{z\}\times\C^r} \to \eval{p^*\omega_{SF}}{\{z\}\times\C^r}$ in $C^{\infty}_{\textup{loc}}(\{z\}\times\C^r)$-topology, and so
$$e^t\eval{p^*\omega_t}{\pi^{-1}(z)} \to \eval{p^*\omega_{SF}}{\pi^{-1}(z)}$$
in $C^{\infty}(\pi^{-1}(z))$-topology. It completes the proof of (ii) since $\eval{p^*\omega_{SF}}{\pi^{-1}(z)}$ is a flat metric for each $z \in \Sigma$.
\end{proof}

\section{Remarks}

The totally collapsing case of $\Sigma$ being a point 
(and so $c_1(X)=0$) is considered in H.D. Cao's work \cite{Cao85} 
on the Ricci flow proof of the Calabi-Yau Theorem. The convergence 
of flow metric to the point metric is certainly in very strong sense, and coincides with our scenario.

We briefly describe a possible approach to adjust the previous argument to the 
general situation allowing singular fibres. We use the same setting as in \cite{To10} 
as described below, and would stick to the existing notations in 
the current work.

\vspace{0.1in}

In the general case, the smooth fibration $\pi: X\to \Sigma$ is replaced by a 
holomorphic map $F: X\to Y$ between complex manifolds with the 
image $\Sigma=F(X)$ being possibly singular. In practice, this map 
is generated by the line bundle $mK_X$ for some large positive 
integer $m$ and this manifold $Y$ is some complex projective 
space $\mathbb{CP}^N$. 

There is a subvariety $S$ in $X$ with the restriction of $F$ to $X
\setminus S\to \Sigma\setminus F(S)$ being a submersion. Now 
$\omega_{\Sigma}=\omega_Y|_\Sigma$ for some K\"ahler metric 
$\omega_Y$ over $Y$.

We still consider the collapsing case of $\dim_{\mathbb{C}}X=n
>n-r=\dim_{\mathbb{C}}Y$, and then the restricted $F$ gives a 
smooth bundle over $\Sigma\setminus F(S)$ of fibre dimension 
$r$.

As in \cite{To10}, there is a smooth function $H$ over $X$ 
defined by 
$$\omega^{n-r}_\infty\wedge \omega^r_0=H\omega^n_0$$ 
which vanishes exactly at $S$ and is locally comparable with a 
(finite) sum of the squares of the norms of holomorphic functions. 
Furthermore, one can have another smooth real non-negative 
function $\sigma$ over $Y$ vanishing exactly at $F(s)$. Obviously 
we have  
$$\sqrt{-1}\p\sigma\wedge\bar\p\sigma\leq C\omega_Y, ~~
-C\omega_Y\leq \sqrt{-1}\p\bar\p\sigma\leq C\omega_Y.$$ 

We would also use $\sigma$ to denote its pull-back on $X$. 

\vspace{0.1in}

Now we consider the arguments in the previous sections in this 
general situation. 

Lemma 2.1 is still valid by the recent work \cite{ST11} by Song-Tian, and 
so is (2.4). Thus Lemma 2.3 still holds.

The estimate in Lemma 2.4 needs to be replaced by 
$$|e^t(\varphi_t-\Phi_t)|\leq Ce^{B\sigma^{-\lambda}}$$ 
over $X\setminus S$ for some positive constants $C$, $B$ and 
$\lambda$. The exact same argument works except that at the 
end where the Poincar\'e constant and Green's function bound 
would no longer be uniform, resulting in the degeneracy of the 
estimate. Please see \cite{To10} for detail.
 
\vspace{0.1in}

For the Maximum Principle argument in Section 3, in the same spirit as \cite{ST07}, we
consider the term $\widetilde Q=e^{-B\sigma^{-\lambda}}\cdot Q$. 

Clearly, $\nabla \widetilde Q=e^{-B\sigma^{-\lambda}}\nabla Q +Q
\nabla e^{-B\sigma^{-\lambda}}$, and so $\nabla Q=e^{B\sigma^
{-\lambda}}\nabla \widetilde Q+BQ\nabla \sigma^{-\lambda}$. In this 
work, $\nabla$ means $\p$ and $(\cdot, \cdot)$ is the Hermitian product with respect to the flow metric $\omega_t$. Then we have the following computation,
\begin{equation}
\begin{split}
\square \widetilde Q
&= e^{-B\sigma^{-\lambda}}\square Q-Q \Delta e^{-B\sigma^
{-\lambda}}-2{\rm Re}\(\nabla Q, \nabla e^{-B\sigma^{-\lambda
}}\) \\
&= e^{-B\sigma^{-\lambda}}\square Q-Q \(-Be^{-B\sigma^
{-\lambda}}\Delta \sigma^{-\lambda}+B^2e^{-B\sigma^
{-\lambda}}|\nabla \sigma^{-\lambda}|^2\) \\
&~~~~~~~~  -2{\rm Re}\(e^{B\sigma^{-\lambda}}\nabla 
\widetilde Q+BQ\nabla \sigma^{-\lambda}, -B e^{-B\sigma^
{-\lambda}}\nabla \sigma^{-\lambda}\) \\
&= e^{-B\sigma^{-\lambda}}\square Q+2B{\rm Re}\(\nabla 
\widetilde Q, \nabla \sigma^{-\lambda}\) \\
&~~~~~~~~ +BQ e^{-B\sigma^{-\lambda}}\Delta \sigma^
{-\lambda}+B^2Qe^{-B\sigma^{-\lambda}}|\nabla \sigma^
{-\lambda}|^2. \nonumber
\end{split}
\end{equation}

The following useful estimates can be established by the properties of 
$\sigma$ summarized earlier. 
\begin{equation}
\begin{split}
|\nabla \sigma^{-\lambda}|^2
&= \lambda^2\sigma^{-2\lambda-2}|\nabla \sigma|^2 \\
&= \lambda^2\sigma^{-2\lambda-2}\Tr_{\omega_t}(\sqrt{-1}
\p\sigma\wedge\bar\p\sigma) \\
&\leq C\sigma^{-2\lambda-2}\Tr_{\omega_t}(C\omega_
\infty) \\
&\leq C\sigma^{-2\lambda-2}, \nonumber
\end{split}
\end{equation}
\begin{equation}
\begin{split}
|\Delta\sigma^{-\lambda}|
&\leq \lambda\sigma^{-\lambda-1}|\Delta\sigma|+
|\lambda(\lambda+1)\sigma^{-\lambda-2}|\nabla\sigma|
^2| \\
&\leq \lambda\sigma^{-\lambda-1}|\Tr_{\omega_t}
(\sqrt{-1}\p\bar\p\sigma)|+\lambda(\lambda+1)\sigma^
{-\lambda-2} \Tr_{\omega_t}(\sqrt{-1}\p\sigma\wedge
\bar\p\sigma)\\
&\leq C\sigma^{-\lambda-2}\Tr_{\omega_t}(C\omega_
\infty) \\
& \leq C\sigma^{-\lambda-2}.\nonumber
\end{split}
\end{equation}

Meanwhile, the lower bound for $\Delta \Phi_t$ is replaced by the degenerate term $-C\sigma^{-\mu}$.

Combining all these, we have 
\begin{equation}
\begin{split}
\square\widetilde Q
&\leq e^{-B\sigma^{-\lambda}}\(CAe^t\sigma^{-\mu}
+CAe^t+(C-A)\Tr_{\omega_t}\omega_0\)+2B{\rm Re}\(\nabla 
\widetilde Q, \nabla \sigma^{-\lambda}\) \\
&~~~~~~~~ +CB |Q| e^{-B\sigma^{-\lambda}}\sigma^
{-\lambda-2}+CB^2 |Q| e^{-B\sigma^{-\lambda}}\sigma^
{-2\lambda-2}. \nonumber
\end{split}
\end{equation}

Now we apply Maximum Principle to get an upper bound for the 
term $$\widetilde Q=e^{-B\sigma^{-\lambda}}\cdot Q=e^{-B
\sigma^{-\lambda}}\cdot\bigl(\log\Tr_{\omega_t}\omega_0-t-
Ae^t(\varphi_t-\Phi_t)\bigr).$$

Clearly, we only need to consider the case $Q > 0$ at the point being considered. 
Then at the maximum value point in the region $X\times 
[0, S]$ with $t>0$ (which is clearly not in $S$), we have 
$$0\leq \(CAe^t\sigma^{-\mu}+CAe^t+(C-A)\Tr_
{\omega_t}\omega_0\)+CBQ\sigma^{-\lambda-2}+
CB^2Qe^{-2\lambda-2}$$
Again, we take a sufficiently large $A$ such that $C-A<-1$. Using 
$$Q=\log\Tr_{\omega_t}\omega_0-t-Ae^t(\varphi_t-\Phi_t)
\leq \log\Tr_{\omega_t}\omega_0-t+CAe^{B\sigma^
{-\lambda}},$$ we end up with
$$\Tr_{\omega_t}\omega_0\leq C\sigma^{-2\lambda-2}
\log\Tr_{\omega_t}\omega_0+Ce^te^{(B+\epsilon)\sigma^
{-\lambda}}$$ 
for some $\epsilon>0$. So we have
$$\Tr_{\omega_t}\omega_0\leq Ce^te^{(B_0+\epsilon)
\sigma^{-\lambda}},$$ 
from which we conclude
$$\widetilde Q\leq C.$$

Hence we have $e^{-t}\omega_0\leq F(\sigma)
\omega_t$ which is a degenerate analogue of (3.5). By the 
same argument as in Section 3, one conclude that 
$\frac{1}{G(\sigma)}\hat\omega_t\leq\omega_t
\leq G(\sigma)\hat\omega_t$, indicating the metric collapses 
along fibres $\pi^{-1}(z)$ for $z \in \Sigma \setminus F(S)$. 

For the discussion in Section 4, the general case is essentially more 
involved. For example, the complex Monge-Amp\`ere equation in the 
definition of $\omega_{GKE}$, 
$$(\omega_\Sigma+\sqrt{-1}\p\bar\p u)^{n-r}=Fe^u
\omega^{n-r}_\Sigma,$$
is now over a (possibly) singular variety $\Sigma$. One could pull 
it back to the desingularization of $\Sigma$, and the results in 
\cite{DP10, EGZ09, Z06} give a bounded weak solution which is 
also continuous by \cite{Z06}. However, in order to apply the 
argument as in \cite{ST07} for the flow convergence, one needs 
sufficient regularity away from $S$. Fortunately, this has been 
done explicitly in \cite{ST08}, where the local uniform convergence 
at the level of metric potential away from $S$ is also achieved. 
Combining with the local collapsing (in fact just bound) of flow 
metric, we have the local convergence in $C^{1, \alpha<1}$-norm 
away from $S$. 

The discussion in Section 5 is local, as primarily in the original 
work of \cite{GTZ11}, and so all the conclusions in Section 5 
are valid in the local sense. 

\vspace{0.1in}

For the general case, the convergence so far is only local which 
brings little control on the global geometry. The global control 
remains to be an interesting problem. Nonetheless, we know 
that the scalar curvature on the whole manifold is uniformly 
bounded. See \cite{Z09} for the non-collapsing case, and 
\cite{ST11} for the general case including the collapsing case. 

\bibliographystyle{amsalpha}
\bibliography{../citations}

\providecommand{\bysame}{\leavevmode\hbox to3em{\hrulefill}\thinspace}
\providecommand{\MR}{\relax\ifhmode\unskip\space\fi MR }
\providecommand{\MRhref}[2]{%
  \href{http://www.ams.org/mathscinet-getitem?mr=#1}{#2}
}
\providecommand{\href}[2]{#2}
\begin{thebibliography}{Fon11b}

\bibitem[Ca]{Cao85}
Huai-Dong Cao, \emph{Deformation of {K}\"ahler metrics to {K}\"ahler-{E}instein
  metrics on compact {K}\"ahler manifolds}, Invent. Math. \textbf{81} (1985),
  no.~2, 359--372. \MR{799272 (87d:58051)}

\bibitem[Ch]{Ch75}
Shiu-Yuen Cheng, \emph{Eigenvalue comparison theorems and its geometric
  applications}, Math. Z. \textbf{143} (1975), no.~3, 289--297. \MR{0378001 (51
  \#14170)}

\bibitem[CY]{ChY75}
Shiu-Yuen Cheng and Shing-Tung Yau, \emph{Differential equations on {R}iemannian
  manifolds and their geometric applications}, Comm. Pure Appl. Math.
  \textbf{28} (1975), no.~3, 333--354. \MR{0385749 (52 \#6608)}

\bibitem[DP]{DP10}
Jean-Pierre Demailly and Nefton Pali, \emph{Degenerate complex
  {M}onge-{A}mp\`ere equations over compact {K}\"ahler manifolds}, Internat. J.
  Math. \textbf{21} (2010), no.~3, 357--405. \MR{2647006 (2012e:32039)}

\bibitem[EGZ]{EGZ09}
Philippe Eyssidieux, Vincent Guedj, and Ahmed Zeriahi, \emph{Singular
  {K}\"ahler-{E}instein metrics}, J. Amer. Math. Soc. \textbf{22} (2009),
  no.~3, 607--639. \MR{2505296 (2010k:32031)}

\bibitem[Eh]{Eh51}
Charles Ehresmann, \emph{Les connexions infinit\'esimales dans un espace
  fibr\'e diff\'erentiable}, Colloque de topologie (espaces fibr\'es),
  {B}ruxelles, 1950, Georges Thone, Li\`ege, 1951, pp.~29--55. \MR{0042768
  (13,159e)}

\bibitem[Ev]{Ev82}
Lawrence~C. Evans, \emph{Classical solutions of fully nonlinear, convex,
  second-order elliptic equations}, Comm. Pure Appl. Math. \textbf{35} (1982),
  no.~3, 333--363. \MR{649348 (83g:35038)}

\bibitem[FG]{FG65}
Wolfgang Fischer and Hans Grauert, \emph{Lokal-triviale {F}amilien kompakter
  komplexer {M}annigfaltigkeiten}, Nachr. Akad. Wiss. G\"ottingen Math.-Phys.
  Kl. II \textbf{1965} (1965), 89--94. \MR{0184258 (32 \#1731)}

\bibitem[F1]{F1}
Frederick Tsz-Ho Fong, \emph{{K\"ahler-Ricci} flow on projective bundles over
  {K\"ahler-Einstein} manifolds}, to appear in Trans. Amer. Math. Soc., April
  2011, arXiv:1104.3924v1 [math.DG].

\bibitem[F2]{F2}
\bysame, \emph{On the collapsing rate of the {K\"ahler-Ricci} flow with
  finite-time singularity}, December 2011, arXiv:1112.5987.

\bibitem[Gi]{Gil12}
Matthew Gill, \emph{Collapsing of products along the {K}\"ahler-{R}icci flow},
  March 2012, arXiv:1203.3781 [math.DG].

\bibitem[GT]{GT}
David Gilbarg and Neil~S. Trudinger, \emph{Elliptic partial differential
  equations of second order}, Classics in Mathematics, Springer-Verlag, Berlin,
  2001, Reprint of the 1998 edition. \MR{1814364 (2001k:35004)}

\bibitem[GTZ]{GTZ11}
Mark Gross, Valentino Tosatti, and Yuguang Zhang, \emph{Collapsing of abelian
  fibred calabi-yau manifolds}, Duke Math. J. \textbf{162} (2013), no.~3, 517--551.

\bibitem[Kr]{Kr83}
Nikolai.~V. Krylov, \emph{Boundedly inhomogeneous elliptic and parabolic equations in
  a domain}, Izv. Akad. Nauk SSSR Ser. Mat. \textbf{47} (1983), no.~1, 75--108.
  \MR{688919 (85g:35046)}

\bibitem[KT]{KT08}
S{\l}awomir Ko{\l}odziej and Gang Tian, \emph{A uniform {$L^\infty$} estimate
  for complex {M}onge-{A}mp\`ere equations}, Math. Ann. \textbf{342} (2008),
  no.~4, 773--787. \MR{2443763 (2010a:32067)}

\bibitem[LY]{LY80}
Peter Li and Shing-Tung Yau, \emph{Estimates of eigenvalues of a compact
  {R}iemannian manifold}, Geometry of the {L}aplace operator ({P}roc. {S}ympos.
  {P}ure {M}ath., {U}niv. {H}awaii, {H}onolulu, {H}awaii, 1979), Proc. Sympos.
  Pure Math., XXXVI, Amer. Math. Soc., Providence, R.I., 1980, pp.~205--239.
  \MR{573435 (81i:58050)}

\bibitem[MS]{MS73}
J.~H. Michael and Leon~M. Simon, \emph{Sobolev and mean-value inequalities on
  generalized submanifolds of {$R^{n}$}}, Comm. Pure Appl. Math. \textbf{26}
  (1973), 361--379. \MR{0344978 (49 \#9717)}

\bibitem[S]{Siu87}
Yum-Tong Siu, \emph{Lectures on {H}ermitian-{E}instein metrics for stable
  bundles and {K}\"ahler-{E}instein metrics}, DMV Seminar, vol.~8, Birkh\"auser
  Verlag, Basel, 1987. \MR{904673 (89d:32020)}

\bibitem[SSW]{SSW11}
Jian Song, G\'abor Sz\'ekelyhidi, and Ben Weinkove, \emph{The {K\"ahler-Ricci}
  flow on projective bundles}, Int. Math. Res. Not. (2013), no.~2, 243--257. \MR{3010688}

\bibitem[SeT]{SeT08}
Natasa Sesum and Gang Tian, \emph{Bounding scalar curvature and diameter along
  the {K}\"ahler-{R}icci flow (after {P}erelman)}, J. Inst. Math. Jussieu
  \textbf{7} (2008), no.~3, 575--587. \MR{2427424 (2009c:53092)}
  
\bibitem[ST1]{ST07}
Jian Song and Gang Tian, \emph{The {K}\"ahler-{R}icci flow on surfaces of
  positive {K}odaira dimension}, Invent. Math. \textbf{170} (2007), no.~3,
  609--653. \MR{2357504 (2008m:32044)}

\bibitem[ST2]{ST08}
\bysame, \emph{Canonical measures and kahler-ricci flow}, J. Amer. Math. Soc. \textbf{25} (2012), no.~2, 303-353. \MR{2869020}

\bibitem[ST3]{ST11}
\bysame, \emph{Bounding scalar curvature for global solutions of the
  {K\"ahler-Ricci} flow}, November 2011, arXiv:1111.5681 [math.DG].

\bibitem[SW1]{SW09}
Jian Song and Ben Weinkove, \emph{The {K\"ahler-Ricci} flow on {Hirzebruch}
  surfaces}, J. Reine Angew. Math. \textbf{659} (2011), 141--168. \MR{2837013 (2012g:53142)}

\bibitem[SW2]{SW_KRF}
\bysame, \emph{Lecture notes on the {K\"ahler-Ricci} flow}, December 2012, arXiv:1212.3653 [math.DG]

\bibitem[To]{To10}
Valentino Tosatti, \emph{Adiabatic limits of {R}icci-flat {K}\"ahler metrics},
  J. Differential Geom. \textbf{84} (2010), no.~2, 427--453. \MR{2652468
  (2011m:32039)}

\bibitem[TZ]{TZ06}
Gang Tian and Zhou Zhang, \emph{On the {K}\"ahler-{R}icci flow on projective
  manifolds of general type}, Chinese Ann. Math. Ser. B \textbf{27} (2006),
  no.~2, 179--192. \MR{2243679 (2007c:32029)}

\bibitem[Y]{Y78}
Shing-Tung Yau, \emph{On the {R}icci curvature of a compact {K}\"ahler manifold
  and the complex {M}onge-{A}mp\`ere equation. {I}}, Comm. Pure Appl. Math.
  \textbf{31} (1978), no.~3, 339--411. \MR{480350 (81d:53045)}

\bibitem[Z1]{Z06}
Zhou Zhang, \emph{On degenerate {M}onge-{A}mp\`ere equations over closed
  {K}\"ahler manifolds}, Int. Math. Res. Not. (2006), Art. ID 63640, 18.
  \MR{2233716 (2007b:32058)}

\bibitem[Z2]{Z09}
\bysame, \emph{Scalar curvature bound for {K}\"ahler-{R}icci flows over minimal
  manifolds of general type}, Int. Math. Res. Not. \textbf{20} (2009),
  3901--3912. \MR{2544732 (2010j:32038)}

\bibitem[Z3]{Z10}
\bysame, \emph{Scalar curvature behavior for finite-time singularity of
  {K}\"ahler-{R}icci flow}, Michigan Math. J. \textbf{59} (2010), no.~2,
  419--433. \MR{2677630 (2011j:53128)}

\end{thebibliography}
\end{document}